\documentclass[12pt]{amsart}
\usepackage{amssymb}
\usepackage{mathrsfs}

\theoremstyle{plain}

\newtheorem{theorem}{Theorem}
\newtheorem{corollary}{Corollary}[theorem] 
\newtheorem{lemma}{Lemma}
\newtheorem{lemmaX}{Lemma}

\newtheorem{problem}{Problem}
\newtheorem{remark}{Remark}
\theoremstyle{definition}

\newcommand{\eps}{\varepsilon}

\newcommand{\dist}{{\rm dist\,}}

\newcommand{\spn}{{\rm span\,}}

\newcommand{\DD}{\mathcal D}

\newcommand{\LL}{\mathscr L}

\newcommand{\R}{\mathbb R}

\newcommand{\NN}{\mathbb N}

\newif\ifComplain
\Complaintrue           
\def\complain#1{\ifComplain\ifhmode \newline\fi{\sf *** \ \ #1
\\}\fi}

\newif\ifmarglab
\makeatletter
\def\label#1{\@bsphack\ifmarglab\marginpar{LAB:#1}\fi\if@filesw {\let\thepage\relax
   \def\protect{\noexpand\noexpand\noexpand}%
   \edef\@tempa{\write\@auxout{\string
      \newlabel{#1}{{\@currentlabel}{\thepage}}}}%
   \expandafter}\@tempa
   \if@nobreak \ifvmode\nobreak\fi\fi\fi\@esphack}
\makeatother
\marglabfalse

\usepackage{graphicx}

\long\def\onefigure#1#2{
\begin{figure*}[tbp]
\begin{center}
#1
\end{center}
\caption{#2}
\end{figure*}
} 

\newcommand\newipefig[2]
{\onefigure{\includegraphics{#1.pdf}}{\label{#1} #2}}

\begin{document}

\title[Projections and greedy approximation]{Alternating projections, remotest projections, and greedy approximation}

\author[P. A. Borodin]{Petr A. Borodin}
\address{Department of Mechanics and
Mathematics, Moscow State University, Moscow 119991, Russia}
\email{pborodin@inbox.ru}
\author[E. Kopeck\'a]{ Eva Kopeck\'a}
\address{Department of Mathematics\\
   University of Innsbruck\\
 A-6020 Innsbruck, Austria}
\email {eva.kopecka@uibk.ac.at}

\thanks{The first author was supported by the grant of the Government of the
Russian Federation (project 14.W03.31.0031).}



\begin{abstract}

Let $L_1,L_2,\dots,L_K$ be a family of closed subspaces of a Hilbert space $H$, $L_1\cap \dots \cap L_K =\{0\}$; let $P_k$ be the orthogonal projection onto $L_k$. We consider two types of consecutive projections of an element $x_0\in H$: alternating projections $T^nx_0$, where $T=P_K\circ\dots\circ P_1$,  and remotest projections $x_n$ defined recursively, $x_{n+1}$ being the remotest point for $x_n$ among $P_1x_n,\dots,P_Kx_n$. These $x_n$ can be interpreted as residuals in greedy approximation  with respect to a special dictionary associated with $L_1,L_2,\dots,L_K$. We establish parallels  between convergence properties separately known for alternating projections, remotest projections, and greedy approximation in $H$.

Here are some results. If $L_1^\perp+\dots+L_K^\perp=H$, then $x_n\to 0$ exponentially fast.   In case $L_1^\perp+\dots+L_K^\perp\not=H$, the convergence $x_n\to 0$ can be arbitrarily slow for certain $x_0$.   Such a dichotomy, exponential rate of convergence everywhere on $H$, or arbitrarily slow convergence for certain starting elements,   is valid  for greedy approximation with respect to general dictionaries. The dichotomy was known for alternating projections. Using the methods developed for greedy approximation we  prove  that  $|T^nx_0|\le C(x_0,K)n^{-\alpha(K)}$ for certain positive $\alpha(K)$
and all starting points $x_0\in L_1^\perp+\dots+L_K^\perp$.

\end{abstract}

\maketitle


\section*{Introduction}

Let $H$ denote a real Hilbert space with norm $|\cdot|$ and scalar product $\langle\cdot,\cdot\rangle$.
Let $K\ge 2$ be a fixed integer number and let $L_1,L_2,\dots,L_K$ be a family of $K$
closed   subspaces of $H$ such that $L_1\cap \dots \cap L_K =\{0\}$.
Let $x_0\in H$ and $k_1,k_2,\dots \in \{1,\dots,K\}$
be an arbitrary sequence. Consider the sequence of vectors $x_n$ defined by  $x_n=P_{k_n}x_{n-1}$,
where $P_k$ denotes the orthogonal projection of $H$ onto the subspace  $L_k$.

In the case where $\{k_n\}$ contains each $k\in \{1,\dots,K\}$ infinitely often, $\{x_n\}$ is a weakly-null sequence according to~\cite{AA}. If $H$ is infinite-dimensional and $K\ge 3$, then the sequence $\{x_n\}$  does not, in general, converge in norm~\cite{P,KM,KP}. A list of various conditions sufficient for the norm convergence of $\{x_n\}$ can be found in~\cite{K}.

The most studied special case of  $\{k_n\}$ is the cyclic sequence
$k_n=n \mod K$.  For $T=P_K\circ\dots\circ P_1$   the  {\em alternating projections}
\begin{equation}\label{altpr}
T^n x_0, \qquad n=1,2,\dots,
\end{equation}
 enjoy the following convergence properties:
\begin{enumerate}
\item[(A1)] {\em $|T^nx_0|\to 0$ for any $x_0\in H$;}
\item[(A2)]{\em if $L_1^\perp+\dots+L_K^\perp=H$, then $\|T^n\|\le q^n$ and thus $|T^nx_0|\le |x_0| q^n$ for certain $q=q(L_1,\dots,L_K)\in [0,1)$;}
\item[(A3)]{\em if $L_1^\perp+\dots+L_K^\perp\not=H$, then for any sequence $\alpha_n\to 0$ there exists a starting point $x_0\in H$ such that $|T^nx_0|\ge \alpha_n$, $n=1,2,\dots$.}
\end{enumerate}
For $K=2$ the convergence property (A1) is a classical result of von Neumann \cite{N}, and for  $K\geq 2$ of Halperin \cite{Ha}.
The {\em dichotomy result} (A2) and (A3)  was obtained independently in \cite{BDH,DH} and \cite{BaGM1,BaGM2}.

Another natural way of consecutively projecting onto $L_1,\dots,L_K$ is to choose in every step the projection of $x_n$ which is the  nearest to the origin, or, equivalently, which is the remotest from $x_n$.  Namely, for any $x_0\in H$ we consider the sequence $x_n$ of its {\it remotest projections} defined inductively by
\begin{equation}\label{rempr}
x_{n+1}=P_{i(n)}x_n, \qquad n=0,1,2,\dots,
\end{equation}
where  the  (possibly not unique) number $i(n)$ is chosen so that
$$
\dist(x_n,L_{i(n)})=\max\{\dist(x_n,L_k): k\in \{1,\dots,K\}\}.
$$
Remotest projections have been   investigated in \cite{GPR}, \cite{BB}, and \cite{BarRZ} in a more general setting of convex sets $L_k$.  When $L_k$'s are closed subspaces, Theorem~5.3 of~\cite{BB} provides $x_n\to 0$ for any $x_0\in H$ under the additional assumption that $L_1^\perp+\dots+L_K^\perp=H$.

In the von Neumann case $K=2$, alternating projections and remotest projections are almost the same object: $T^nx_0$ is the $(2n-1)$-th remotest projection of $P_1x_0$. Consequently, analogues of (A1)--(A3) are valid for remotest projections in this case. A natural question arises:
\smallskip

\noindent{\em Do remotest projections (\ref{rempr}) satisfy  something like (A1)--(A3) for any $K\ge 2$?}
\smallskip

In order to answer it,  we observe that remotest projections can be interpreted as residuals in a special greedy approximation process.
We recall the notion of greedy approximation with respect to a dictionary (see \cite{teml_book} for a detailed survey).

A subset $D$ of the unit sphere $S(H)=\{s\in H: |s|=1\}$ is called a {\it dictionary} if $\overline{\spn} D=H$.
For any dictionary $D\subset S(H)$ and any $x_0\in H$, the {\it pure greedy
algorithm} (PGA) generates a sequence $x_n$ defined inductively
by
\begin{equation}\label{greal}
x_{n+1}=x_n-\langle x_n,g_{n+1}\rangle g_{n+1}, \qquad n=0,1,2,\dots,
\end{equation}
where the element $g_{n+1}\in D$ is such that
$$
|\langle x_n,g_{n+1}\rangle |=\max \{|\langle x_n,g\rangle|: g\in D\}.
$$
The existence of the above maximum is an additional condition on
$D$.

It is easy to see that $x_{n+1}= x_n-y_{n+1}$, where $y_{n+1}$ is one of the nearest points to $x_n$ in the set $\Lambda(D)=\{\lambda g:\, \lambda \in {\mathbb R},\, g\in D\}$. Thus the existence of $g_{n+1}$ in (\ref{greal}) is equivalent to the proximality of $\Lambda(D)$.  When $\Lambda(D)$ is not proximal, a {\em weak greedy algorithm} (WGA) is an option. The sequence $\{x_n\}$ is again defined recursively by (\ref{greal}), but $g_{n+1}$ is such that
$|\langle x_n,g_{n+1}\rangle |\ge t_{n+1} \sup \{|\langle x_n,g\rangle|: g\in D\}$, for a given sequence $t_n\in (0,1)$ of weakness parameters (see~\cite{teml_book} for details).

We observe  that (\ref{rempr}) coincides with (\ref{greal}) for
$$
D=D_L=(L_1^\perp\cup\dots\cup L_K^\perp)\cap S(H).
$$
The set $D_L$ is indeed a dictionary, since
   $L_1\cap \dots \cap L_K =\{0\}$  and
$$
\overline{\spn} D_L=\overline{\spn} (L_1^\perp\cup\dots\cup L_K^\perp)=\overline{L_1^\perp+\dots+L_K^\perp}=H.
$$
Denote by $P_k^\perp$  the orthogonal projection onto $L_k^\perp$.
The remotest projection $P_{i(n)}x_n$ clearly corresponds to the  projection $P_{i(n)}^\perp x_n$
which is the nearest to $x_n$ among $P_1^\perp x_n, \dots, P_K^\perp x_n$:
$$
x_{n+1}=P_{i(n)}(x_n)=x_n-P_{i(n)}^\perp x_n.
$$

Jones proved  that PGA  converges for every dictionary $D$, that is,  $|x_n|\to 0$   for $\{x_n\}$ defined by (\ref{greal}) and any initial element $x_0\in H$  (\cite{Jones}; see also \cite[Ch. 2]{teml_book}).
Since $D_L$ is   a dictionary,   remotest projections (\ref{rempr}) have a property similar to (A1):
\begin{remark}\label{remark1}
$|x_n|\to 0$ for any $x_0\in H$; \\ here
$\{x_n\}$  is the sequence of remotest projections defined by (\ref{rempr}).
\end{remark}

DeVore and Temlyakov singled out a set of starting points generated by the dictionary   for which  the greedy algorithm converges polynomially fast \cite{DeVT}; see our Section~3 for details.
Consequently,  remotest projections as greedy residuals   converge polynomially fast for starting points from $L_1^\perp+\dots+L_K^\perp$. In particular:

\smallskip
(R4) {\it If $x_0\in L_1^\perp+\dots+L_K^\perp$, then $|x_n|\le C(x_0) n^{-1/6}$.}
\smallskip

\noindent  Hence in the von Neumann case of  $K=2$ an analog of (R4) is valid for alternating projections  as well.  A natural   question
arises, if this is valid for any $K\ge 2$:
\smallskip

{\it Do alternating projections (\ref{altpr}) converge polynomially fast for starting points from $L_1^\perp+\dots+L_K^\perp$?}

\smallskip

Deutsch and Hundal formulate a similar conjecture
in~\cite[Remark 6.5]{DH}   without any connection to greedy approximation.

In this paper we investigate the interplay  between alternating projections, remotest projections, and  greedy approximation.

If $L_1^\perp+\dots+L_K^\perp=H$, remotest projections indeed converge fast.  According to
Remark~\ref{remark2},
$|x_n|\le |x_0|r^n$ for a certain $r=r(L_1,\dots, L_K)<1$. Moreover,
this $r$ is less than the best known $q$ in (A2), according to  Remark~\ref{remark3}.

In Theorem~\ref{theorem2} we give a
sufficient condition for the fast convergence of greedy algorithm: every element of $H$ is a finite linear combination of elements from the dictionary. The dichotomy of (A2)--(A3) type is valid also for greedy approximation with respect to general dictionaries  and hence for remotest projections as well, as we show in Theorem~\ref{theorem1} and Corollary~\ref{corollary1.1}.

Alternating projections indeed satisfy an analogue of (R4): if $x_0\in L_1^\perp+\dots+L_K^\perp$, then $|T^nx_0|\le C(x_0,K)n^{-\alpha(K)}$ for certain positive $\alpha(K)$, and $\alpha(2)=1/2$ happens to be the best possible; see Theorem~\ref{theorem4} and  Theorem~\ref{theorem3}.

We conclude by verifying that in spite of many similar convergence properties the family of remotest projections  is really distinct  from alternating projections. In Theorem~\ref{theorem5}
we give an example of remotest projections that never become cyclic.

\section{Fast convergence of remotest projections}\label{remotest}

In this section we investigate when  remotest projections converge fast and establish an analogue of (A2) for them.

\smallskip

\begin{remark}\label{remark2}
Let $L_1^\perp+\dots+L_K^\perp=H$. Then
$$
|x_n|\le |x_0| (1-\rho^2)^{n/2},
$$
where $\{x_n\}$ is the sequence of remotest projections defined in (\ref{rempr}), and
$$
\rho=\rho(L_1,\dots,L_K)=\inf \{ \max_k \dist(x,L_k): x\in S(H)\}>0.
$$
More precisely,
\begin{equation}\label{rho*}
|x_n|\le |x_0| (1-\rho^2)^{1/2} (1-\rho_*^2)^{(n-1)/2},
\end{equation}
where
$$
\rho_*=\inf \{ \max_k \dist(x,L_k): x\in
S(H)\cap (L_1\cup\dots\cup L_K)\}\ge \rho.
$$
\end{remark}
\begin{proof} Since $L_1^\perp+\dots+L_K^\perp$ is   closed,  $\rho>0$ by~\cite[Theorem~5.19]{BB}; see also Lemma~\ref{lemma1} proved in the next section. By Pythagoras' theorem and the definition of $\rho$,
$$
|x_{n+1}|^2=|x_n|^2-\dist(x_n,L_{i(n)})^2 \le |x_n|^2(1-\rho^2),
$$
hence $|x_{n}| \le |x_0| (1-\rho^2)^{n/2}$.
The refinement (\ref{rho*}) follows from   $x_n\in L_1\cup\dots\cup L_K$ for $n\ge 1$: the norm
decreases with coefficient at least $(1-\rho_*^2)^{1/2}$ after the
first projection.
\end{proof}

We will generalize Remark~\ref{remark2}  to  greedy approximation with respect to general dictionaries in Theorem~\ref{theorem2}.

Clearly, $\rho(L_1,\dots,L_K)<1$ for any non-trivial family
$L_1,\dots,L_K$. On the other hand, if, for example, the $L_k$'s  are mutually orthogonal, then  $\rho_*=1$. How large exactly  $\rho$ can be seems not  to be  known.

\begin{problem}\label{problem1}
Calculate
$\rho_K:=\sup \rho(L_1,\dots,L_K)$,
where  the supremum is taken over all families $L_1,\dots,L_K\subset H$.
\end{problem}

It is easy to see that $\rho(2)=1/\sqrt{2}$.

Next we compare the rate of convergence of  remotest
projections $|x_n|$ with that of alternating projections $|T^n(x_0)|$ in the ``(A2) case" when $L_1^\perp+\dots+L_K^\perp=H$.

The best known estimate for $\|T^n\|$, which is also valid only when $L_1^\perp+\dots+L_K^\perp=H$, is  \cite[Theorem
4.4]{BaGM2}:
\begin{equation}\label{normT^n}
\|T^n\|\le\left(1-\left(\frac{1-c}{4K}\right)^2\right)^{n/2},
\qquad n=1,2,\dots,
\end{equation}
where $c=c(L_1,\dots,L_K)$ denotes the generalized {\em Friedrichs number}
$$
c:=\sup\left\{\frac{\sum_{j\not=k}\langle y_j,y_k\rangle}{(K-1)(|y_1|^2+\dots+|y_K|^2)}:\,
y_j\in L_j,\, \sum_{k=1}^K |y_k|^2\not=0\right\}.
$$

We have to compare $\rho_*$ with $(1-c)/(4K)$.

\begin{remark}\label{remark3}
For any family $L_1,\dots,L_K$,
$$
\rho_*\ge \frac{1-c}{K-1},
$$
and thus (\ref{rho*}) witnesses a   faster  rate of convergence  than (\ref{normT^n}), in spite of   $T$ involving $K$ projections instead of just one.
\end{remark}
\begin{proof} We choose $y\in S(H)\cap (L_1\cup\dots\cup L_K)$
such that $\max_k \dist(y,L_k)=\rho_*$. If there is no such $y$, we
take one for which this equality ``nearly" holds. We may assume
$y\in L_1$. For $y_1=y$ and $y_k=P_k(y)$ if $k\in\{2,\dots,K\}$, we have
\begin{equation}\notag
\begin{split}
1-c&\leq 1-
\frac{\sum_{j\not=k}\langle y_j,y_k\rangle}{(K-1)(|y_1|^2+\dots+|y_K|^2)}=
\frac{\sum_{k=1}^K\sum_{j\not=k}(|y_k|^2-\langle y_j,y_k\rangle)}{(K-1)(|y_1|^2+\dots+|y_K|^2)} \\
&=\frac{\sum_{j\not=1}\langle y_1,y_1-y_j\rangle+\sum_{k=2}^K\sum_{j\not=k}\langle y_k,y_k-y_j\rangle}{(K-1)(|y_1|^2+\dots+|y_K|^2)}.
\end{split}
\end{equation}
Since $\langle y_j,y_1-y_j\rangle=0$ for  $j\in\{1,\dots,K\}$ we can continue estimating by
\begin{equation}\notag
\begin{split}
\ \ \ &=\frac{\sum_{j\not=1}\langle y_1-y_j,y_1-y_j\rangle+\sum_{k=2}^K\sum_{j\not=k,j\not=1}\langle y_k,y_1-y_j\rangle}{(K-1)(|y_1|^2+\dots+|y_K|^2)} \\
&\leq
\frac{(K-1)\rho_*^2+(K-1)(K-2)\rho_*}{(K-1)(1+|y_2|^2+\dots+|y_K|^2)}\le
\rho_*(\rho_*+K-2)\le (K-1)\rho_*.
\end{split}
\end{equation}
\end{proof}

Above we have compared only {\em estimates} for rates of convergence of different projections but {\em not the rates} themselves.

In several particular examples of   $K$-tuples $\LL=\{L_1,\dots,L_K\}$ and starting elements $x_0$ remotest projections indeed do converge faster than alternating projections.  A quantitative or a category result of this sort for   tuples $(\LL,x_0)$ when  $K\ge 3$   would be of interest.
Of course, $\{T^n(x_0)\}$ may converge to zero faster than $\{x_n\}$ for
particular $x_0$'s.  Consider the four 1-dimensional subspaces $L_1,\dots,L_4$ of $\R^2$, generated by the vectors $(1,0)$, $(0,1)$, $(1,1)$, $(1,\varepsilon-1)$. Here  $\eps>0$ is a small positive number. For $x_0\in L_3\setminus \{0\}$, we have
$$
x_1\in L_4, x_2\in L_3, x_3\in L_4, x_4\in L_3, \dots,
$$
and $x_n\not= 0$ for all $n$, since $L_4$ is the remotest subspace for elements of $L_3$ and vice versa and these two subspaces are not mutually orthogonal. At the same time, $T^nx_0=0$ for all $n=1,2,\dots$, since already $P_2P_1x_0=0$ due to the orthogonality of $L_1$ and $L_2$.

\section{Dichotomy for greedy approximation}\label{greedydichotomy}

In this section, we present a dichotomy result of (A2)--(A3) type for the pure and the weak greedy algorithms 
 and hence also for the remotest projections.

For a dictionary $D\subset S(H)$, we define
$$
\rho(D)=\inf_{x\in S(H)} \sup\{|\langle x,g\rangle| :\, g \in  D\}.
$$
With a family of closed subspaces $L_1,\dots,L_K$ we associate the dictionary
$$
D_L=(L_1^\perp\cup\dots\cup L_K^\perp)\cap S(H).
$$
It is easy to see that   $\rho(D_L) =\rho(L_1,\dots, L_K)$ as  defined in Remark~\ref{remark2}.

Characteristics of dictionaries similar to $\rho(D)$ have  already been used in greedy approximation theory; see {\em e.g.} \cite{teml}. We show that $\rho(D)=0$ if and only if the dictionary $D$ is contained in an ``arbitrarily thin board".

\begin{lemma}\label{lemma1}
Let $D\subset S(H)$ be a dictionary.
\begin{enumerate}
\item[(a)] The equality $\rho(D)=0$ holds if and only if there exists an orthonormal sequence $\{w_n\}$ in $H$ so that
$$
\lim_{n\to \infty}\sup\{|\langle w_n,g\rangle| :\, g \in  D\}=0.
$$
\item[(b)] If $D=D_L$ for  closed  subspaces  $L_1,\dots,L_K$  of $H$, then  $\rho(D_L)=0$ is equivalent to $L_1^\perp+\dots+L_K^\perp\not= H$.
\end{enumerate}
\end{lemma}
\begin{proof} If $\rho(D)=0$ we choose a weakly convergent sequence $v_k\in S(H)$ so that $\lim_{k\to \infty}\sup\{|\langle v_k,g\rangle| :\, g \in  D\}=0$.  The sequence $\{v_k\}$ converges weakly to zero, since $\overline{\spn} D=H$.  There is an orthonormal sequence $\{w_n\}$ and a subsequence of $\{v_k\}$, so that
$\lim_{n\to \infty}|w_n-v_{k_n}|=0$ (see {\em e.g.} Lemma~6.2 of~\cite{K}). Then $
\lim_{n\to \infty}\sup\{|\langle w_n,g\rangle| :\, g \in  D\}=0$.
The opposite implication of (a) is obvious.

That in the situation of (b) the existence of the orthonormal sequence $\{w_n\}$ is equivalent to  $L_1^\perp+\dots+L_K^\perp\not= H$ was proved  in  Lemma~1.1 of \cite{K} which in turn follows from \cite{BB}.

We give here a different proof.  Assume $L_1^\perp+\dots+L_K^\perp\not= H$. The unit ball $B_k$ of $L_k^\perp$ is a weakly compact set, hence $C=B_1+\dots+B_K$ is a symmetric  weakly compact  convex set with empty interior. We choose a sequence $\{v_n\}$  that separates vectors of
an  arbitrarily small norm from $C$.  Namely, for $n\in \NN$ we choose $z_n\in H\setminus C$ and $ v_n\in H$ so that
$|z_n|\leq 1/n$, $|v_n|=1$, and  $\max_{x\in C}|\langle v_n,x\rangle|\leq  \langle v_n,z_n\rangle\leq 1/n$. Hence $\lim_{n\to \infty}\sup\{|\langle v_n,g\rangle| :\, g \in  D_L\}=0=\rho(D_L)$, since $D_L\subset C$.

Now assume $L_1^\perp+\dots+L_K^\perp=H$ and to the contrary assume that $\rho(D_L)=0$. By  (a) there exists an orthonormal sequence
$\{w_n\}$  such that
$\sup_{g\in D_L}|\langle w_n,g\rangle|< 1/n^2$ for $n\in \NN$.
 We choose $\lambda_j\in \R$  and $g_j\in L_j^\perp$ so that
 $$
 x=\sum_{n=1}^{\infty}w_n/n=\lambda_1g_1+\dots+\lambda_Kg_K.
 $$
 Then for all $n\in \NN$ we have
 $$
 \frac 1n=|\langle w_n,x\rangle|\leq \sum_{j=1}^K|\lambda_j||\langle w_n,g_j\rangle|\le\frac 1{n^2}\sum_{j=1}^K|\lambda_j|,
 $$
 which is a contradiction.
\end{proof}

 The characteristic $\rho(D)$ influences the rate of convergence of the greedy algorithm.
 If $\rho(D)>0$, the algorithm converges fast everywhere; if $\rho(D)=0$,
 it converges arbitrarily slowly for certain starting elements.

\begin{theorem}\label{theorem1}
Let $D\subset S(H)$ be a dictionary.
\begin{enumerate}
\item[(i)]  If $\rho(D)>0$, then
\begin{equation}\label{WGArate}
|x_n|\le |x_0| \prod_{k=0}^{n-1}(1-t_k^2\rho(D)^2)^{1/2},
\qquad n\in \NN,
\end{equation}
for every $x_0\in H$ and its sequence  $\{x_n\}$ of  WGA greedy residuals with weakness parameters $\{t_k\}$ as defined in (\ref{greal}) . In particular, for PGA greedy residuals (if PGA is possible for $D$) we have $|x_n|\le |x_0| (1-\rho(D)^2)^{n/2}$.
\item[(ii)]   If $\rho(D)=0$, then  for every sequence $\alpha_n\to 0$ there exists a starting element $x_0\in H$ such that its sequence  of greedy residuals in PGA or in WGA with any weakness parameters satisfies  $|x_n|\geq \alpha_n$ for $n\in \NN$.
\end{enumerate}
\end{theorem}
\begin{proof} (i) According to the definition (\ref{greal}) of $\{x_n\}$,
\begin{equation}\notag
\begin{split}
|x_{k+1}|^2&=|x_k|^2-|\langle x_k, g_{k+1}\rangle |^2 \le |x_k|^2-t_k^2 \sup_{g\in D}|\langle x_k, g\rangle |^2\\
&\le |x_k|^2-t_k^2|x_k|^2\rho(D)^2= |x_k|^2(1-t_k^2\rho(D)^2),
\end{split}
\end{equation}
and hence (\ref{WGArate}) holds.

(ii) We can assume that $\alpha=\max_{m\in \NN} |\alpha_m|\leq 1/2$:
if $\alpha>1/2$ and $x_0\in H$ works for the sequence $\{\alpha_m/(2\alpha)\}$, then $2\alpha x_0$ works for $\{\alpha_m\}$.
We choose  $1\leq m_1<m_2<\dots$ so that $|\alpha_m|<1/(2n+2)$ if $m>m_n/(4n)$.

By Lemma~\ref{lemma1}, there is an orthonormal sequence  $\{w_n\}$ such that
$$
\sup\{|\langle w_n,g\rangle|:\, g\in  D\}\leq 1/m_n, \qquad n\in \NN.
$$
Consider $x_0=\sum_{n=1}^{\infty}w_n/n$. Then $|x_0|=\pi/\sqrt6<2$.
The $m$-th greedy residual of $x_0$ has the form
$$
x_m=x-\lambda_1g_1-\dots-\lambda_mg_m,
$$
where $g_j\in D$ and $|\lambda_j|=|\lambda_j g_j| \le |x_{j-1}|\leq |x_0|$.

 For a given $m\in \NN$ we choose $n\in \NN$ so that
$$
\frac1{2(n+1)}\leq |\alpha_m|\leq \frac1{2n}.
$$
Then $m\leq m_n/(4n)$, and
$$
|x_m|\geq |\langle x_m,w_n\rangle|\geq 1/n- \sum_{j=1}^m |\lambda_j| |\langle g_j,w_n\rangle|
$$
$$
\geq 1/n -|x_0|m/m_n\geq 1/(2n)\geq |\alpha_m|.
$$
\end{proof}

Since remotest projections correspond to greedy residuals, Theorem~\ref{theorem1} and (b) of Lemma~\ref{lemma1} imply the  dichotomy below.
This is a remotest projections analogue of the dichotomy result (A2)--(A3) for the alternating projections.

\begin{corollary}\label{corollary1.1}
Let $L_1,\dots, L_K$ be closed subspaces of a Hilbert space $H$.
\begin{enumerate}
\item[(i)]  If $L_1^\perp+\dots+L_K^\perp= H$, then there exists $\rho\in(0,1]$ such that $|x_n|\le |x_0|(1-\rho^2)^{n/2}$
for every $x_0\in H$ and its sequence of remotest projections (\ref{rempr}).
\item[(ii)]   If $L_1^\perp+\dots+L_K^\perp\not= H$, then  for every sequence $\alpha_n\to 0$ there exists a starting element $x_0\in H$ such that its remotest projections satisfy  $|x_n|\geq \alpha_n$ for $n\in \NN$.
\end{enumerate}
\end{corollary}

The equality $L_1^\perp+\dots+L_K^\perp= H$ means that every element of $H$ can be represented as a linear combination of $K$ elements of the dictionary $D_L=(L_1^\perp\cup\dots\cup L_K^\perp)\cap S(H)$ associated with the remotest projections. The statement (i) of Corollary~\ref{corollary1.1} can be generalized to arbitrary dictionaries in this sense.

\begin{theorem}\label{theorem2}
Let a dictionary $D\subset S(H)$ be so that every element of $H$ is a finite linear
combination of  elements of $D$. Then $\rho(D)>0$ and  the estimate (\ref{WGArate}) holds for every starting point $x_0\in H$.
\end{theorem}
\begin{proof} According to Theorem~\ref{theorem1} it is enough to prove that $\rho(D)>0$.  We mimic  the proof   of   Lemma~\ref{lemma1}(b).
Assume  to the contrary  $\rho(D)=0$. By Lemma~\ref{lemma1} there exists an orthonormal sequence
$\{w_n\}$  such that
$$
\sup_{g\in D}|\langle w_n,g\rangle|< 1/n^2,\ n\in \NN.
$$

 We choose $\lambda_j\in {\mathbb R}$  and $g_j\in D$ so that
 $$
 x=\sum_{n=1}^{\infty}w_n/n=\lambda_1g_1+\dots+\lambda_Ng_N.
 $$
 Then for all $n\in \NN$ we have
 $$
 \frac 1n=|\langle w_n,x\rangle|\leq \sum_{j=1}^N|\lambda_j||\langle w_n,g_j\rangle|\le\frac 1{n^2}\sum_{j=1}^N|\lambda_j|,
 $$
 which is impossible.
\end{proof}

The converse  is not true: $\rho(D)>0$ does not imply
that every $x\in H$ can be represented
  as a finite linear
combination of  elements of $D$. Take, for example, $D\subset S(l_2)$ consisting of all unit vectors with finite number of non-zero coordinates.

A normalized Hamel basis of $H$ is an example of a dictionary that represents every   $x\in H$ as a  finite linear combination of its elements. If $H$ is infinite dimensional, the number of these elements is not uniformly bounded.

\begin{remark}\label{wclosed} Let  $D\subset S(H)$  be a dictionary such  that every $x\in H$  is a   linear combination of  finitely many elements of $D$.
Suppose, moreover,  that the set $\Lambda(D)=\{\lambda g:\, \lambda \in {\mathbb R},\, g\in D\}$ is weakly closed. Then  there exists $K\in \NN$ so that  each   $x\in H$  is  a linear combination of  $K$ elements of $D$.
\end{remark}
\begin{proof} Let $B$ be the closed unit ball of $H$.
The set $A=B\cap \Lambda(D)$ is symmetric and weakly compact, hence
$$
A_n=\underbrace{A+\dots+A}_n
$$
is also symmetric and weakly compact.
Since $H=\bigcup_{n=1}^{\infty}A_n$ is a countable union of closed
sets $A_n$, by the Baire category theorem there is an $N\in \NN$ so that the interior of $A_N$ is not empty.
Since $A_N$ is symmetric, the origin is contained in  the interior of   $A_N+A_N$. Hence
 every element of $H$ is a linear combination of no more than $K=2N$ elements of $D$.
\end{proof}

\section{Convergence rate for starting points from $L_1^\perp\cup\dots\cup L_K^\perp$}

Let $D\subset S(H)$ be a dictionary for which PGA (\ref{greal}) works. The general greedy approximation theory guarantees the rate of convergence
\begin{equation}\label{1/6}
|x_n|\le \frac{C(x_0)}{n^{1/6}}
\end{equation}
of greedy residuals for starting elements $x_0\in A_1(D):=\bigcup_{\lambda>0} \DD_\lambda$, where
\begin{equation}\notag
\DD_\lambda=\overline{\left\{\sum_{k=1}^m \lambda_k g_k: g_k\in D, m\in {\mathbb N}, \sum_{k=1}^m |\lambda_k|\le \lambda\right\} };
\end{equation}
see~\cite{DeVT}, \cite[Theorem 2.18]{teml_book}.  Moreover, the power $1/6$ here can be replaced by $0.182$~\cite{Sil} but cannot be replaced by $0.1898$~\cite{Liv}. The   exact power  in (\ref{1/6}) is not known.

Consider the special case of remotest projections (\ref{rempr}), when $D=D_L=(L_1^\perp\cup\dots\cup L_K^\perp)\cap S(H)$.
Denote $Y=L_1^\perp+\dots+L_K^\perp$ and  by $B_k$ the unit ball of $L_k ^{\perp}$. Then $C=B_1+\dots+B_K$ is a weakly compact set.
By the triangle inequality $\DD_\lambda\subset \lambda C$ for $\lambda>0$ and
$$
Y\subset A_1(D_L)=\bigcup_{\lambda>0} \DD_\lambda\subset\bigcup_{\lambda>0} \lambda C=Y.
$$
Hence for $x_0\in Y$ and its remotest projections $x_n$ the inequality   (\ref{1/6}) holds. This is the property (R4)   mentioned in Introduction. Since we deal with a very specific dictionary $D_L$, we face

\begin{problem}\label{problem2}
Can one refine (\ref{1/6}) in the case of remotest
projections (\ref{rempr})?
\end{problem}

In the von Neumann case $K=2$ the answer is yes: in Theorem~\ref{theorem3} we show that the best possible power is $1/2$. For
     $K\ge 3$ the problem  is open.

In Theorem~\ref{theorem3} and Theorem~\ref{theorem4} we
prove estimates of the type (\ref{1/6}) for the norm of alternating projections $|T^nx_0|$ of an element $x_0\in Y$. We use   a machinery developed in~\cite{DeVT}:   a simplified version in Theorem~\ref{theorem3} and a more complicated   version in Theorem~\ref{theorem4}.

We begin with preliminary lemmata and a notation.

\begin{lemma}\label{lemma2}
If $y\in Y=L_1^\perp+\dots+L_K^\perp$, then $P_j y\in Y$ for any $j\in \{1,\dots,K\} $.
\end{lemma}
\begin{proof}
Let $y=y_1+\dots+y_K$, $y_i\in L_i^\perp$. We have
$$
P_jy=y-P_j^\perp y= y_1+\dots+y_{j-1}+(y_j-P_j^\perp y)+y_{j+1}+\dots+y_K\in Y,
$$
since $y_j-P_j^\perp y\in L_j^\perp$.
\end{proof}

For $y\in Y=L_1^\perp+\dots+L_K^\perp$, we denote
\begin{equation}\label{s(y)}
s(y)= \inf \{|y_1|+\dots+|y_K|: y=y_1+\dots+y_K, y_j\in L_j^\perp\}.
\end{equation}
It is readily checked that $s$ is a norm on $Y$.
By the
triangle inequality $|y|\leq s(y)$  for $y\in Y$, hence every norm-open set is also $s$-open. For completeness we observe, although we do not use  it in this paper, that the norm $s$ is complete.

\begin{remark}\label{remark5}
The subspace $Y$ equipped with the norm $s$ is a Banach space.
\end{remark}
\begin{proof}
Denote by $(\tilde Y,s)$  the completion of $Y$.  The identity mapping $I:(Y,s)\to(H,|\cdot|)$ is  Lipschitz.
It admits a unique uniformly continuous extension, an injection  $f:(\tilde Y,s)\to (H,|\cdot|)$; see {\em e.g.} \cite{R}, p. 82.
It remains to show that $f(\tilde Y)=Y$.
Assume $y_n\in Y$ and $s$-$\lim_{n\to \infty} y_n=\tilde y\in \tilde Y$. Since $\{y_n\}$ is an $s$-Cauchy sequence in $Y$  it is norm-Cauchy   as well, and  $\lim_{n\to \infty} y_n=y\in H$.
The sequence  $\{y_n\}$ is contained in the weakly compact set
$rB_{L^{\perp}_1}+\dots+rB_{L^{\perp}_K}\subset Y$ for some $r>0$ since it is bounded in $s$, implying $y\in Y$. The continuity of $f$ yields
$f(\tilde y)=f(s$-$\lim_{n\to \infty} y_n)=\lim_{n\to \infty} f(y_n)=\lim_{n\to \infty} y_n=y\in Y$.
\end{proof}

For any $x\in H$, let $g(x)\in D_L$  be the vector in $D_L$ with direction  closest to that of $x$ (if there is more than one, we choose one of them):
\begin{equation}\label{g(x)}
\langle x,g(x)\rangle=\max \{\langle x,g\rangle: g\in  D_L\}.
\end{equation}
Denote  the cosine of the angle between   $x$ and $g(x)$ by
\begin{equation}\label{rho(x)}
\rho(x)=\frac{\langle x,g(x)\rangle}{|x|}.
\end{equation}
The direction  $g(x)$  determines the subtrahend in the remotest step (\ref{rempr}): $x_{n+1}=x_n-\langle x_n, g(x_n)\rangle g(x_n)$. The value  $\rho(x_n)$ determines the decay of the norm in the remotest step:
\begin{equation}\label{rho}
|x_{n+1}|^2=|x_n|^2-\langle x_n,g(x_n)\rangle^2=|x_n|^2(1-\rho (x_n)^2).
\end{equation}
On the subspace $Y$ the ratio between the Hilbert space norm and the norm $s$ gives a handy lower estimate for $\rho$.

\begin{lemmaX}\label{lemmaA}~\cite[Lemma 2.17]{teml_book}
For $y\in Y=L_1^\perp+\dots+L_K^\perp$, we have
$$
\rho(y)\ge \frac{|y|}{s(y)}.
$$
\end{lemmaX}
\begin{proof}  Let $y=y_1+\dots+y_K$, $y_i\in L_i^\perp$. Then
$$
|y|^2=\langle y, y_1+\dots+y_K\rangle = \sum_{i=1}^K |y_i| \langle y, y_i/|y_i|\rangle
$$
$$
\le \sum_{i=1}^K |y_i| \langle y,g(y)\rangle= |y| \rho(y) \sum_{i=1}^K |y_i|,
$$
so that
$$
\rho(y)\ge \frac{|y|}{|y_1|+\dots+|y_K|}.
$$
\end{proof}
In the special case of the dictionary $D_L=(L_1^\perp\cup L_2^\perp)\cap S(H)$ the rate of convergence (\ref{1/6}) of the greedy approximation can be improved to $O(n^{-1/2})$.

\begin{theorem}\label{theorem3}
Let $L_1,L_2$ be closed subspaces of $H$. Then for any $x_0\in Y=L_1^\perp+L_2^\perp$, we have
$$
|x_n|\le \frac{C(x_0)}{\sqrt{n}}, \qquad n\in \NN.
$$
for remotest projections (\ref{rempr}) of $x_0$, and
\begin{equation}\label{t3}
|T^n(x_0)|\le \frac{\tilde C (x_0)}{\sqrt{n}}, \qquad n\in\NN.
\end{equation}
for alternating projections (\ref{altpr}) of $x_0$, where $C(x_0)$ and $\tilde C (x_0)$ are constants depending only on $x_0$.

In both of these estimates, $\sqrt{n}$ cannot be replaced by $n^{1/2+\varepsilon}$ for any $\varepsilon>0$.
\end{theorem}

\begin{proof} 1. By Lemma~\ref{lemma2}, both of the  sequences $\{x_n\}$ and $\{T^n x_0\}$ belong to $Y$ for any $x_0\in Y$.
We show that the sequence $\{s(x_n)\}$
is decreasing, hence
$s(x_{n})\le s(x_1)$ for $n\in \NN$.  Indeed, every $x_n$ belongs to $L_1$ or to $L_2$. Suppose $x_n\in L_1$, $x_n=y_1+y_2$, $y_i\in L_i^\perp$. We have $x_n\perp y_1$ and $y_2=x_n-y_1$, hence $|y_1|\leq |y_2|$. Next, $x_{n+1}=P_2(x_n)=P_2(y_1)=y_1+y_2'$, where $y_2'=P_2(y_1)-y_1=-P_2^\perp y_1\in L_2^\perp$, and thus $|y_2'|\le |y_1|$. Consequently,
$$
|y_1|+|y_2'|\le |y_1|+|y_1| \le |y_1|+|y_2|,
$$
and hence $s(x_{n+1})\le s (x_n)$.

2. By Lemma~\ref{lemmaA} for any $x_0\in Y$ and $n\in \NN$ we have
$$
\rho(x_n)\ge \frac{|x_n|}{s(x_n)}\ge \frac{|x_n|}{s(x_1)},
$$
which together with (\ref{rho}) implies
\begin{equation}\label{x_n}
|x_{n+1}|^2\le |x_n|^2\left(1-\frac{|x_n|^2}{s(x_1)^2}\right), \qquad n\in \NN.
\end{equation}

Now we need

\begin{lemmaX}\label{lemmaB}~\cite[Lemma 2.16]{teml_book}.
Suppose  the sequence $\{c_n\}_{n=1}^\infty$ satisfies $c_n\ge 0$, $c_1\le A$, and  $c_{n+1}\le c_n(1-c_n/A)$ for $n\in \NN$. Then
$$
c_n\le \frac{A}{n}, \qquad n\in \NN.
$$
\end{lemmaX}
\begin{proof} For  $n=1$ the inequality is  satisfied;
for $n=2$ it is proved as follows:
$$
c_2\le c_1 (1-c_1/A)\le \max_{t\in {\mathbb R}} \{f(t):=t(1-t/A)\}=f(A/2)=A/4\le A/2.
$$
For $n\ge 3$ the inequality is proved by induction using the monotonicity of $f$ on $[0,A/2]$:
$$
c_{n+1}\le c_n(1-c_n/A)=f(c_n)\le f(A/n)=\frac{A}{n}\left(1-\frac{1}{n}\right)<\frac{A}{n+1}.
$$
\end{proof}

Applying Lemma~\ref{lemmaB} to (\ref{x_n}) and taking into account the inequality $|x_1|^2\le s(x_1)^2$, we get
$$
|x_n|^2\le \frac{s(x_1)^2}{n}, \qquad n\in \NN.
$$
Since $T^n (x_0)$ is the $(2n-1)$-th remotest projection of $P_1x_0$, the last inequality implies
$$
|T^n (x_0)|^2\le \frac{s(P_1x_0)^2}{2n-1}, \qquad n\in \NN,
$$
and the first part of the theorem is proved.

3. Finally we show the optimality of the estimate: for any $\varepsilon>0$ we present two subspaces $L_{1},L_{2}$ and an element $x_0\in Y=L_1^\perp+L_2^\perp$ such that $|x_n|>C/n^{1/2+\varepsilon}$ for some $C>0$.

Consider $H$ as a sum of mutually orthogonal 2-dimensional Euclidean subspaces $H_m$, $m\in \NN$. In each $H_m$, we take unit vectors $e_m^1$ and $e_m^2$ with the angle $\alpha_m=1/m$ between them and also unit vectors $y_m^1\perp e_m^1$, $y_m^2\perp e_m^2$ with the angle $\pi-\alpha_m$ between $y_m^1$ and $y_m^2$. Let $L_j$ be the closed subspace of $H$ generated by $e_1^j,e_2^j,e_3^j,\dots$, so that $L_j^\perp$ is the closed linear span of $y_1^j,y_2^j,y_3^j,\dots$, for  $j=1,2$.

Setting $y_m=y_m^1+y_m^2$ ($m\in \NN$), consider
$$
x_0=\sum_{m=1}^\infty \frac{y_m}{m^{1/2+\varepsilon}} =\sum_{m=1}^\infty \frac{y_m^1}{m^{1/2+\varepsilon}} + \sum_{m=1}^\infty \frac{y_m^2}{m^{1/2+\varepsilon}}\in L_1^\perp+L_2^\perp;
$$
all series converge in $H$. We have
$$
|x_0|^2=\sum_{m=1}^\infty \frac{|y_m|^2}{m^{1+2\varepsilon}}=4\sum_{m=1}^\infty \frac{\sin^2 (\alpha_m/2)}{m^{1+2\varepsilon}}.
$$

Consecutive application of the projections $P_1,P_2,P_1,P_2,\dots$ to $x_0$ occurs ``coordinatewise".  That is, we iterate the projections in each of the 2-dimensional  subspace $H_m$, where the term $y_m/m^{1/2+\varepsilon}$ is consecutively projected onto lines with directions $e_m^1, e_m^2, e_m^1, e_m^2, \dots$, and its length is multiplied by $\cos (\alpha_m/2)$ after the first projection and then each time by $\cos \alpha_m$. Thus
\begin{equation}\notag
\begin{split}
&|T^nx_0|^2=|x_{2n}|^2=\sum_{m=1}^\infty \frac{|y_m|^2}{m^{1+2\varepsilon}}\cos^2\frac{\alpha_m}{2} (\cos \alpha_m)^{2(2n-1)} \\
&=\sum_{m=1}^\infty \frac{\sin^2 \alpha_m}{m^{1+2\varepsilon}}(\cos \alpha_m)^{4n-2}
\ge \sum_{m=1}^\infty \frac{((2 \alpha_m)/\pi)^2}{m^{1+2\varepsilon}}\left(1- \frac{\alpha_m^2}{2}\right)^{4n-2} \\
&=\frac{4}{\pi^2}\sum_{m=1}^\infty \frac{1}{m^{3+2\varepsilon}}\left(1- \frac{1}{2m^2}\right)^{4n-2}
\ge \frac{4}{\pi^2}\sum_{m\ge \sqrt{2n-1}} \frac{1}{m^{3+2\varepsilon}}\left(1- \frac{1}{2m^2}\right)^{2m^2} \\
&\ge C_1\sum_{m\ge \sqrt{2n-1}}\frac{1}{m^{3+2\varepsilon}}\ge C_2\int_{\sqrt{2n}}^\infty\frac{dt}{t^{3+2\varepsilon}}=\frac{C_3}{n^{1+\varepsilon}}
\end{split}
\end{equation}
for certain constants $C_{1,2,3}>0$. Above we have used that $2\sin \alpha\cos \alpha=\sin 2\alpha$,
that $\sin \alpha\geq 2\alpha /\pi$ and $\cos \alpha \geq 1-\alpha^2/2$ for $\alpha\in [0,\pi/2]$, and that $\lim_{t\to +\infty} (1-1/t)^t=1/e$.
\end{proof}

 Assume that the dense set  $Y=L_1^{\perp}+\dots +L_K^{\perp}$ is not  closed and hence (A3) takes place. Deutsch and Hundal asked where     the initial points
for the arbitrarily slow convergence of the alternating projections $T$ lie, and conjectured that they lie in   $H\setminus Y$ \cite{DH}.

 For a sequence $r=\{r_n\}$, $r_n\geq 0$, $r_n\to 0$,  let
$S_r=\{x \in H:\, |T^n x|>r_n$  for all $n\}$
be the starting points of ``$r$-slow" convergence of $T$.
M\"uller announced   that
  $S_r\cap (H\setminus Y)\neq \emptyset$ for any $r$ (V. M\"uller, {\em unpublished manuscript}, 2017).
  In the next theorem we resolve the question of Deutsch and Hundal fully. We show  that there even exist sequences $r$ so that $S_r\subset H\setminus Y$.

\begin{theorem}\label{theorem4}
Let  $L_1,\dots,L_K$ be closed subspaces of $H$, and  $Y=L_1^\perp+\dots+L_K^\perp$.  Then for every  $x_0\in Y$ there exists $c(x_0)>0$ such that
\begin{equation}\label{t4}
|T^n(x_0)|\leq c(x_0)\cdot n^{-1/(4 K\sqrt K+2)}, \qquad n=1,2,\dots .
\end{equation}
\end{theorem}
\begin{proof} 1. For $y\in Y$, we have $Ty\in Y$ by Lemma 2. More precisely,
\begin{equation}\label{v}
\begin{split}
Ty&=P_K\dots P_1 y=(I-P_K^{\perp})P_{K-1}\dots P_1 y \\
&= y-v_1-\dots -v_K,
\end{split}
\end{equation}
where $v_j=P_j^\perp P_{j-1}\dots P_1y\in L_j^\perp$.
For $y\neq 0$ we denote
$$
\nu(y)=(|v_1|^2+\dots+|v_K|^2)^{\frac 12}/|y|.
$$
By the Pythagoras theorem,
$|P_{j}\dots P_1y|^2=|P_{j-1}\dots P_1y|^2-|v_j|^2$. Adding these equalities yields
\begin{equation}\label{1step}
|Ty|^2=|y|^2(1-\nu^2(y)).
\end{equation}

2. We   estimate the growth of the norm $s$ defined in (\ref{s(y)}) as $y\in Y$ is being mapped to  $Ty$.
Since $v_j\in Y$, by (\ref{v}) and by Cauchy-Schwarz inequality we obtain
\begin{equation}\label{sTs}
s(Ty)\leq s(y)+\sum_{j=1}^K|v_j| \leq s(y)+\sqrt K|y|\nu(y).
\end{equation}

3. The quantity $\nu(y)$  determines  the decay  of the norm for the alternating projection step $y\to Ty$, just as $\rho(x)$ from (\ref{rho(x)}) does it for the remotest step. We have to estimate $\nu(y)$ from below, similarly  as we estimated $\rho(y)$   in Lemma A.

By (\ref{g(x)}) and (\ref{rho(x)}), one of the nearest points for $y$ in $L_1^\perp \cup\dots\cup L_K^\perp$ is
$$
P_k^\perp y=\langle y, g(y) \rangle g(y) =\rho(y) |y| g(y)
$$
for some $k=k(y)\in \{1,\dots,K\}$. We have
$$
|P_k^{\perp}y|=|y-P_ky|= \dist(y,L_k)\leq |y-P_k\dots P_1y|=|v_1+\dots +v_k|.
$$
On the other hand, Lemma~\ref{lemmaA} gives
$$
|P_k^{\perp}y|=\rho(y) |y|\geq\frac{|y|^2}{s(y)}.
$$
Hence  there exists an $m\in \{1,\dots,k\}$ so that $|v_m|\ge |y|^2/(ks(y))\ge |y|^2/(Ks(y)) $, and thus
\begin{equation}\label{nus}
\nu(y)\geq  \frac{|y|}{Ks(y)}.
\end{equation}

4. Let $0\neq x_0\in Y$ be given. Recursively we define the four sequences
\begin{equation}\notag
\begin{split}
a_n&=|T^nx_0|,\ a_0=|x_0|  \\
\nu_n&=\nu(T^nx_0),\ \nu_0=\nu(x_0)  \\
s_n&=s(T^nx_0),\ s_0=s(x_0)  \\
b_{n+1}&=b_n+\sqrt K a_n\nu_n ,\ b_0=s_0.  \\
\end{split}
\end{equation}
We introduce the auxiliary increasing sequence $\{b_n\}$, as it is not clear that $\{s_n\}$ is monotone.
Then $s_n\leq b_n$ by (\ref{sTs}) and by induction. Hence
$$
\frac{a_n}{b_n}\leq \frac{a_n}{s_n}\leq K\nu_n,
$$
by (\ref{nus}). Therefore
$$
b_{n+1}=b_n(1+\sqrt K \nu_n\frac{a_n}{b_n})\leq b_n(1+K^{\frac 32}\nu_n^{2}).
$$
We define $\alpha=K^{-\frac 32}<1$ and use   Bernoulli's inequality $(1+t)^{\alpha}\leq 1+\alpha t$ for $t\ge 0$ to derive that
$$
b^{\alpha}_{n+1}\leq b_n^{\alpha}(1+\nu_n^2).
$$
Since  $a^2_{n+1}=a^2_n(1-\nu_n^2)$ by (\ref{1step}),
\begin{equation}\label{a/b}
a^2_{n+1}b^{\alpha}_{n+1}\leq a^2_nb^{\alpha}_n(1-\nu_n^4)\leq \dots\leq a^2_0b^{\alpha}_0=|x_0|^2s_0^{\alpha}.
\end{equation}
The sequence $\{b_n\}$ is increasing, hence by (\ref{nus}) we get
$$
\frac{a^2_{n+1}}{b_{n+1}^{2}}\leq\frac{a^2_{n}(1-\nu_n^2)}{b_{n}^{2}}\leq \frac{a^2_{n}}{b_{n}^{2}}\left(1-\frac{1}{K^2}\frac{a^2_n}{s^2_n}\right)\leq \frac{a^2_n}{b^2_n}\left(1-\frac{1}{K^2} \frac{a^2_n}{b^2_n}\right).
$$
Lemma~\ref{lemmaB} then implies
\begin{equation}\label{ab}
a_n^2/b_n^2\leq K^2/(n+1) \leq K^2/n
\end{equation}
for all $n\in \NN$.
Hence by (\ref{a/b}) and (\ref{ab})
$$
a^{4+2\alpha}_n=a^4_nb^{2\alpha}_n\cdot\frac{a^{2\alpha}_n}{b^{2\alpha}_n}\leq |x_0|^4s_0^{2\alpha}K^{2\alpha}n^{-\alpha}
$$
and, finally, for a suitable $c(x_0)>0$ which depends on $x_0\in Y$ only (for given $L_1,\dots, L_K$), we get
$$
a_n\leq |x_0|^{\frac{2}{2+\alpha}}\cdot s(x_0)^{\frac{\alpha}{2+\alpha}}\cdot K^{\frac{\alpha}{2+\alpha}}\cdot n^{-\frac{1}{4/\alpha+2}}=c(x_0)\cdot n^{-\frac{1}{2(2 K^{3/2}+1)}}.
$$
\end{proof}

For $K=2$ Theorem~\ref{theorem4} gives a  much worse estimate than Theorem~\ref{theorem3}, and we face
\begin{problem}\label{problem3}
 Can one improve estimate (\ref{t4}) so that it would give
(\ref{t3}) for $K=2$? Ideally,  find the best possible power in (\ref{t4}) for $K\ge 3$.
\end{problem}

\section{Remotest and alternating projections are distinct}

We have seen that remotest projections are very similar to alternating projections in their convergence properties. A natural question arises: are they basically the same? If  $K=2$  this is obviously the case. Suppose $K\geq 3$.  Does the sequence $i(n)$ in (\ref{rempr}) become cyclic after a while for any starting element $x_0\in H$?

\begin{theorem}\label{theorem5}
There exist three   2-dimensional subspaces $L_{1,2,3}$ of ${\mathbb R}^4$ and a starting element $x_0\in {\mathbb R}^4$ such that the sequence of indices $i(n)$ of  its remotest projections never becomes cyclic.
\end{theorem}
\begin{proof}
Consider ${\mathbb R}^4$ as a sum ${\mathbb R}^2\oplus{\mathbb R}^2$ of two mutually orthogonal 2-dimensional subspaces. We choose unit vectors $e_{1,2,3}$ in the first copy of ${\mathbb R}^2$ such that all three angles between them are acute and $e_3$ ``lies" between $e_1$ and $e_2$. In the second copy of  ${\mathbb R}^2$  we also choose unit vectors $u_{1,2,3}$ with acute angles between them, but now  $u_2$ ``lies" between $u_1$ and $u_3$. Suppose that the angles $\alpha=\widehat{e_1e_2}$, $\beta=\widehat{e_1e_3}$, $\gamma=\widehat{u_1u_3}$, $\delta=\widehat{u_1u_2}$ satisfy the following conditions:
\begin{itemize}
\item[(a)] $\alpha>\gamma>\beta>\delta>\alpha/2$;
\item[(b)] $\cos^2\delta-\cos^2\gamma=\cos^2\beta-\cos^2\alpha$;
\item[(c)] $r_1=\cos^2\alpha/\cos^2\delta$  and $r_2=\cos^2\gamma/\cos^2\beta$ are rational numbers;
\item[(d)]   $r_1^m\not= r_2^n$ for any positive integers $m$ and $n$.
\end{itemize}
For instance, one can take $\alpha=\arccos \sqrt{1/11}$, $\gamma=\arccos \sqrt{2/11}$, $\beta=\arccos \sqrt{3/11}$, $\delta=\arccos \sqrt{4/11}$.

We set $L_j={\rm span}\,\{e_j, u_j\}$, $j=1,2,3$. Since $\beta>\alpha/2$ in view of (a), the distance of any nonzero element $\xi e_3+\eta u_3\in L_3$ from $L_1$ is greater than that from $L_2$. Condition (a) also implies $\delta>\gamma/2$, and hence the distance of any nonzero element $\xi e_2+\eta u_2\in L_2$ from $L_1$ is greater than that from $L_3$. As for elements $x=\xi e_1+\eta u_1\in L_1$, the remotest subspace ($L_2$ or $L_3$) for them depends on the values of $\xi$ and $\eta$. If
$$
|P_2x|^2=|(\xi\cos \alpha) e_2+(\eta\cos \delta) u_2|^2=\xi^2\cos^2\alpha+\eta^2\cos^2\delta
$$
is greater than
$$
|P_3x|^2=|(\xi\cos \beta) e_3+(\eta\cos \gamma) u_3|^2=\xi^2\cos^2\beta+\eta^2\cos^2\gamma,
$$
that is, in view of (b), if  $|\eta|>|\xi|$, then $\rho(x,L_3)>\rho(x,L_2)$; while $|\eta|<|\xi|$ is equivalent to $\rho(x,L_2)>\rho(x,L_3)$.

Thus by iterating remotest projections the  element $x_n=\xi_n e_1+\eta_n u_1\in L_1$ is projected onto $L_3$ if $|\eta_n|>|\xi_n|$ and onto $L_2$ if $|\eta_n|<|\xi_n|$. Its image is  afterwards in both cases projected  onto $L_1$. In other words,
$$
|\eta_n|>|\xi_n| \Rightarrow i(n)=3,\   x_{n+2}=(\xi_n\cos^2\beta) e_1+(\eta_n\cos^2 \gamma) u_1,
$$
$$
|\eta_n|<|\xi_n| \Rightarrow i(n)=2,\   x_{n+2}=(\xi_n\cos^2\alpha) e_1+(\eta_n\cos^2 \delta) u_1.
$$

Consider the starting element $x_0=\xi e_1+\eta u_1\in L_1$, $\xi, \eta>0$, with  an irrational ratio $\xi/\eta$. Then for any even $n=2k$ the element $x_n$ has the form
$$
x_n=\xi_k e_1+\eta_k u_1=(\xi\cos^{2l} \alpha\cos^{2m}\beta) e_1+(\eta\cos^{2l} \delta\cos^{2m}\gamma) u_1
$$
for some $l=l(n)$ and $m=m(n)$ and positive values $\xi_k$ and  $\eta_k$ which are not equal to each other in view of (c). Hence the number $i(n)$ is uniquely determined and for any odd $n$, we have $i(n)=1$. Eventually the whole sequence $\{i(n)\}$ consists of pairs 21 and 31.
We show that this sequence never becomes cyclic.

We consider the sequence of points $(\xi_k, \eta_k)$ in the plane $L_1$. It is changing according to the rule
\begin{equation}\label{i(n)=3}
\xi_{k+1}=\xi_k \cos^2\beta, \eta_{k+1}=\eta_k\cos^2\gamma\quad \hbox{if} \quad |\eta_k|>|\xi_k|,
\end{equation}
\begin{equation}\label{i(n)=2}
\xi_{k+1}=\xi_k \cos^2\alpha, \eta_{k+1}=\eta_k\cos^2\delta\quad \hbox{if} \quad |\eta_k|<|\xi_k|,
\end{equation}
Our aim is to show that the choice between (\ref{i(n)=3}) and (\ref{i(n)=2}) never becomes cyclic.

Consider the sequence $\lambda_k=\ln (\eta_{k}/\xi_{k})$, $k=0,1,2,\dots$. It is changing according to the rule $\lambda_{k+1}=f(\lambda_k)$, where
$$
f(t)=\left\{
\begin{array}{ccc}
t-a &,& t\ge 0\\
t+b &,& t<0,\\
\end{array}
\right.
$$
$a=\ln(\cos^2\beta/\cos^2\gamma)>0$, $b=\ln(\cos^2\delta/\cos^2\alpha)>a$.
This function shifts $[-a,0)$ to $[b-a,b)$ and $[0,b)$ to $[-a,b-a)$. Thus it  transforms the half-interval $[-a,b)$ bijectively onto itself by permutation of  its two parts; it is the  so called ``baker's map". The orbit $\{f^k(t)\}$ of any point finds itself on this half-interval for $k\ge k(t)$. A well-known trick provides a continuous isomorphic  copy of this map. If one glues this half-interval onto a circle $S$ of length $a+b$ by identifying $-a$ and $b$,  then  $f$ induces a rotation $\tilde f$ of $S$ by arc of length $a$. Since the ratio
$$
\frac{a+b}{a}=1+\frac{\ln(\cos^2\delta/\cos^2\alpha)}{\ln(\cos^2\beta/\cos^2\gamma)}
$$
is irrational in view of (d), the orbit $\{\tilde f^k(t)\}$ of any point is dense in $S$. Consequently, $\{\lambda_k\}$ is dense in $[-a,b)$.

This density contradicts the possible cyclicity of $i(2k)$. Indeed, let $i(2k)$ become periodic with period $N$ starting from some $k_0$. Take some $\lambda_k\in (0,a)$ with $k>k_0$, so that $i(2k)=3$ and $i(2k+2)=2$. The sequence $\{\lambda_{k+N\nu}: \nu=0,1,2,\dots\}$ is also dense in $[-a,b)$ by the same  reasoning as above.  Hence  there exists  $\nu$ with $\lambda_{k+N\nu}\in (a,b)$, so that $i(2k+2N\nu)=3$ and $i(2k+2N\nu+2)=3\not= i(2k+2)$, which is a contradiction.
\end{proof}

The behavior of the sequence $\{i(n)\}$ in (\ref{rempr}) seems to be rather
mysterious. We wonder about the following:

\begin{problem}  Assume the sequence
$i(n)\in \{1,2,3\}$ satisfies $i(n)\not= i(n+1)$ for all $n\in \NN$. Do there
exist three closed subspaces  $L_1,L_2,L_3$ of $H$ and a starting point $x_0\in H$ having
exactly  this sequence of indices of its remotest projections (\ref{rempr})?
\end{problem}


\subsection*{Acknowledgements}
We thank V.N. Temlyakov for a fruitful discussion, and David Seifert for explaining the conjecture of Deutsch and Hundal.

\end{document}